\documentclass{lms}
\usepackage{amsmath,amssymb,graphicx,xspace}
\newtheorem{thm}{Theorem}

\newtheorem{lemma}[thm]{Lemma}

\newtheorem{prop}[thm]{Proposition}

\newnumbered{assertion}{Assertion}    
\newnumbered{conjecture}{Conjecture}  
\newnumbered{defn}{Definition}
\newnumbered{hypothesis}{Hypothesis}
\newnumbered{question}{Question}

\newcommand{\barstar}{\mathbin{\bar{*}}}

\newcommand{\Power}{\mathcal{P}}
\newcommand{\N}{\mathbb{N}}
\renewcommand{\phi}{\varphi}
\newcommand{\from}{\colon}
\newcommand{\st}{\,|\,}

\newcommand{\Z}{\mathbb{Z}}
\newcommand{\restr}{\upharpoonright}
\DeclareMathOperator{\Hom}{Hom}
\newcommand{\calA}{\mathcal{A}}

\newcommand{\dikei}{folded kei\xspace}

\title[The quandary of quandles]
 {The quandary of quandles: The Borel completeness of a knot invariant} 

\author{Andrew D. Brooke-Taylor and Sheila K. Miller}

\classno{20N99, 03E15 (primary), 57M27 (secondary)}

\extraline{This work was in large part carried out while both authors were visiting fellows at the Isaac Newton Institute for Mathematical Sciences in the programme `Mathematical, Foundational and Computational Aspects of the Higher Infinite' (HIF).  The first author was supported by the UK Engineering and Physical Sciences Research Council Early Career Fellowship EP/K035703/1, \emph{Bringing set theory and algebraic topology together.} The second author was supported by grants from PSC-CUNY and the City Tech PDAC.}

\begin{document}
\maketitle

\begin{abstract}
The isomorphism type of the knot quandle introduced by Joyce %
is a complete invariant of tame knots.  Whether two quandles are isomorphic is in practice difficult to determine; we show that this question is provably hard: isomorphism of quandles is Borel complete.  The class of tame knots, however, is trivial from the perspective of Borel reducibility, suggesting that equivalence of tame knots may be reducible to a more tractable isomorphism problem.
\end{abstract}

\section{Introduction}
Left distributivity arises in the study of many well-known mathematical objects such as groups, knots and braids, and also in the study of large cardinal embeddings in set theory.
Specifically, left distributive algebras are structures with one binary operation $\ast$ satisfying the left self-distributivity law
$a*(b*c)=(a*b)*(a*c)$.  Familiar examples include the conjugation operation on any group and the implication operation on any Boolean algebra;
symmetric spaces in differential geometry provide further
examples \cite{Ber:GJL}.  The first nontrivial example of a free left distributive algebra on one generator is due to Laver \cite{Lav:LDfree}, who showed that the algebra generated by a certain elementary embedding under the application operation is such an algebra (the existence of these embeddings is one of the strongest known set-theoretic axioms).

Other interesting classes of structures are obtained by adding further algebraic axioms to the left distributive law, with an important case being the quandles. Quandles are left distributive algebras satisfying $a*a=a$ and such that for every $a$ and $c$ in the algebra there is a unique $b$ such that $a*b=c$.  It is quandles that are the focus of this note.  Isomorphism type of quandles is a complete invariant of knots, and we prove that isomorphism of quandles is, from the perspective of Borel reducibility, fundamentally difficult (Borel complete).  After first offering an introduction to quandles and Borel reducibility, we present the technical preliminaries in Section \ref{prelims}, give the main result and corollaries in the next section, and discussion in the final section.

In his doctoral thesis, and published in \cite{Joy:KQ}, Joyce rediscovered quandles and coined the term \emph{quandle}.  There he established many foundational relationships, including those between quandles and group conjugation and quandles and knots.  Indeed he showed that the equational theory of quandles is precisely the equational
theory of the conjugation operation: any identity true in every group
with its conjugation operation is also true in every quandle, and hence
provable from the quandle axioms.  The three quandle axioms may also be viewed as algebraic versions of the familiar Reidemeister moves for translating between different regular projections of equivalent knots. 
One may consequently associate
to any tame knot a quandle generated by the arcs of the knot and with identities dictated by the crossings; notice that all such quandles are finitely presented.  In addition to defining the knot quandle in this way, Joyce showed that these quandles in fact constitute complete invariants for tame knots: two tame knots are equivalent if and only if their associated quandles are isomorphic. (Tame knots essentially correspond to one's intuitive notion of finite knots in three-dimensional space, and in particular are not assumed to be endowed with an
orientation.)

The complexity of classification problems and the study of complete invariants for structures have emerged as major themes in set theory.
Broadly, a classification can be thought to assign mathematical objects of one type --- considered up to isomorphism or some other such equivalence relation --- to mathematical objects of another type (again up to an equivalence relation), where the former act as invariants.  Frequently the objects in question, both those to be classified and the invariants, can be encoded by real numbers.  For example, countable structures with underlying set $\mathbb{N}$, such as groups, rings, and indeed left distributive algebras, can be encoded in a natural way by sets of finite tuples of
natural numbers, and hence by reals. Classification then amounts to finding a reasonably definable map from the reals encoding the structures to the reals encoding the invariants that respects the relevant equivalence relations. Of course, the ``reasonably definable'' is important here --- a non-constructive
proof of the existence of such a map using, say, the Axiom of Choice should
not be considered a classification.
A natural way to exclude such uninformative maps would be to require the map to be continuous, but this interpretation is too restrictive to be practical.
The more liberal
constraint that the map be Borel, however, permits almost all
constructions that arise in practice whilst being restrictive enough to obtain meaningful theorems about the framework.

Classifying structures using Borel maps between sets of encoding reals gives rise to the notion of Borel reducibility.  Given two equivalence relations $E$ and $F$ on real numbers, say that $E$ is Borel reducible to $F$, written $E\leq_B F$, if there is a Borel function $f$ from $\mathbb{R}$ to $\mathbb{R}$ such that for all $x$ and $y$ in $\mathbb{R}$, $x\mathrel{E}y$ holds if and only if $f(x)\mathrel{F}f(y)$ holds.
Establishing that one equivalence relation is
\emph{not} Borel reducible to another has been used in a number of cases
to show that a classification problem is impossible to resolve.
For example, Farah, Toms, and T\"ornquist \cite{FTT:TOECNCA} used this analysis to show
that unital simple separable nuclear $C^*$-algebras are not
classifiable by countable structures (note that each adjective makes the
theorem stronger),
and Foreman, Rudolph, and Weiss \cite{FRW:CPET} showed that ergodic measure-preserving transformations of the unit interval
are not classifiable by countable structures (and indeed much more).
For more on this area see, for example, Hjorth's book \cite{Hjo:COER}. 
Within the scope of knot theory, Kulikov \cite{Kul:NcWK} has recently shown that the class of all knots --- including, for example, wild knots with infinitely many crossings --- is not classifiable by countable structures.

Against this background it is natural to ask: what is the Borel complexity of the isomorphism relation on the most general class of countable left distributive structures, the countable left distributive algebras? This question was indeed posed to the second author by Matt Foreman.  In this note we show that it has the maximum
possible complexity for an isomorphism relation on countable structures:
in the standard terminology introduced in the seminal paper of
Friedman and Stanley \cite{FS:BRCCS},
isomorphism of left distributive algebras is \emph{Borel complete}.
Moreover the same is true for 
the subclasses of racks, quandles, and keis (see Section~\ref{prelims} for definitions).
We show directly that
isomorphism of keis (Definition~\ref{algdefns}.\ref{keidefn}) is
Borel complete; the result for the other, more general classes follows.
We also show that the related class of expansions of left distributive algebras satisfying the set of
axioms Laver \cite{Lav:DAFLDA} denoted by $\Sigma$
(Definition~\ref{LDmonoiddefn})
is Borel complete, although the argument proceeds differently.

Knot theorists express some dissatisfaction with quandles as knot invariants
because of the difficulty in determining whether two quandles are isomorphic.
This difficulty is perhaps not surprising:
our result says that isomorphism of arbitrary countable quandles is
Borel complete.
By contrast, tame knots can reasonably be encoded up to equivalence by
equivalence classes of natural numbers rather than reals, and hence are
trivial in the context of Borel reducibility.
It is therefore reasonable to hope that a complete invariant for knots
that is
simpler than the quandle (in terms of Borel reducibility) might be discovered.
Of course, the subclass of those quandles arising from tame knots is
countable up to quandle isomorphism.
Furthermore, as previously remarked,
all quandles from tame knots are finitely
presented; the class of finitely presented quandles also has only countably
many members up to isomorphism, and so is trivial in Borel reducibility terms.
Finitely presented quandles are thus optimal in this sense as invariants for
tame knots, but their finite presentability is crucial to this fact.
We speculate that a non-Borel complete class of structures with a
definition that does not depend on the cardinality of the presentation of the
structure may provide complete invariants for tame knots which are
in practice easier to test for isomorphism.

\section{Preliminaries}\label{prelims}

As we will be discussing the related classes of left distributive algebras, racks, quandles, and keis, we begin by giving some intuition for them.
These classes of 
structures can usefully be understood in terms of the behaviour of
the action of left multiplication by an element of the algebra.
For structures with underlying set $A$ and binary operation $*$, and for each $a$ in $A$, denote by $m_a$ the map from $A$ to $A$ that
acts by multiplication on the left by $a$, that is, $m_a(b)= a*b$.
Then \emph{left distributive algebras} are those for which
$m_a$ is a
homomorphism from $A$ to itself for each $a$ in $A$.
A \emph{rack} is a left distributive algebra in which
each $m_a$ is an automorphism
(indeed Brieskorn \cite{Bri:ASBS} referred to racks as \emph{automorphic sets}). In a \emph{quandle}, $m_a$ an automorphism and
$a$ is a fixed point of $m_a$ for each $a$ in $A$.
Finally, a \emph{kei} (also called an \emph{involutory quandle}) is a quandle such that each $m_a$ is its own inverse.
The word \emph{quandle} was introduced in 1982 by Joyce \cite{Joy:KQ},
and \emph{kei} in 1943 by Takasaki~\cite{Tak:Kei}, who introduced several variants of keis, many of them reflecting symmetries of geometric configurations of points in the plane.  While together at Cambridge, Wraith and Conway investigated what remains of a group when all the other structure is neglected and only conjugation remains; as a pun on Wraith's name and these wrecked groups, Conway called them \emph{wracks} \cite{Wra:PSK}.
Fenn and Rourke \cite{FeR:RLCT} took this term,
adjusted the spelling to rack, and gave it the present precise meaning.

Formally, these structures can be defined using the following axioms:
\renewcommand{\theenumi}{\roman{enumi}}
\begin{enumerate}
\item\label{LD} For every $a, b,$ and $c$ in $A$, $a*(b*c) = (a*b)*(a*c)$.
\item\label{R} For all $a$ and $c$ in $A$,
there is a unique $b$ in $A$ such that $a*b=c$.
\item\label{Q} For every $a$ in $A$, $a*a=a$.
\item\label{K} For all $a$ and $b$ in $A$, $a*(a*b)=b$.
\end{enumerate}

\renewcommand{\theenumi}{\arabic{enumi}}
\begin{defn}\label{algdefns}
For a set $A$ with one binary operation $\ast$ (an \emph{algebra}),
define:
\begin{enumerate}
\item A \emph{left distributive algebra} is an algebra satisfying axiom (\ref{LD}).
\item A \emph{rack} is an algebra satisfying axioms (\ref{LD}) and (\ref{R}).
\item A \emph{quandle} is an algebra satisfying axioms (\ref{LD}),
(\ref{R}) and (\ref{Q}).
\item\label{keidefn} A \emph{kei} is an algebra satisfying axioms (\ref{LD}),
(\ref{R}), (\ref{Q}) and (\ref{K}).
\end{enumerate}
\end{defn}

There are a number of choices to be made in presenting the above definitions.
Instead of using axiom (\ref{R}),
one can formulate racks using a second operation $\barstar$
such that
the function $m_a:b\mapsto a*b$ is inverse to
the function $b\mapsto a\barstar b$: formallly,
one requires that
for all $a$ and $b$, $a\barstar(a*b)=a*(a\barstar b)=b$ holds.
This has the advantage of eliminating the existential quantifier.
Whether to consider self distributive structures as left distributive, like we do here, or right distributive
(with axioms (\ref{R}) and (\ref{K}) reformulated
for right multiplication) is an arbitrary choice.
Many relevant references on racks, quandles, and keis use right distributivity; we chose left distributivity
in order to easily view these classes of structures as subclasses of
the left distributive algebras.

There is another well-studied left distributive structure, this one with two operations: the left distributive operation $\ast$ and another operation $\circ$ that behaves like composition.
These algebras were first studied by Laver
\cite{Lav:LDfree} as algebras of large cardinal embeddings in which the operation $\circ$ is in fact composition.

\begin{defn}\label{LDmonoiddefn}
We denote by $\Sigma$ the following collection of four identities.
\begin{center}
$\begin{array}{lcl}
a\circ(b\circ c) &= &(a\circ b)\circ c 
\\
(a\circ b)* c &= & a*(b* c) 
\\
a*(b\circ c) &= &(a*b)\circ(a* c) 
\\
(a*b)\circ a &= &a\circ b 
\end{array}$
\end{center}
\end{defn}
Note that left distributivity follows from the second and fourth identities via the equalities
$a*(b*c)=(a\circ b)*c=((a*b)\circ a)*c=(a*b)*(a*c)$.
Dehornoy refers to algebras satisfying $\Sigma$ as
\emph{LD-monoids}; we use Laver's original
phrase ``algebras satisfying $\Sigma$'' to avoid any potential confusion with other uses of ``monoid.''

If $\circ$ is a group operation on $A$ then the fourth equational condition
of $\Sigma$ determines that the
other operation $*$ must be the conjugation operation
$a*b=a\circ b\circ a^{-1}$.
Taking $*$ to denote conjugation in the group in question,
it is straightforward to check that the other identities of $\Sigma$
are also satisfied,
so any group with its multiplication and conjugation operations is an
algebra satisfying $\Sigma$.

Laver showed, among other things, that $\Sigma$ is a conservative extension of the left distributive law \cite{Lav:LDfree}.
Thus any free left distributive algebra may be expanded to a free algebra on the same generators satisfying $\Sigma$: any identity on elements of the free left distributive algebra will hold in the algebra satisfying $\Sigma$ if and only if it is a consequence of the left distributive law.  For more on this, the linearity of several orderings on the free left distributive algebra (from the large cardinal hypothesis), and a normal form for terms in the free left distributive algebra, see \cite{Lav:LDfree} and \cite{Lav:DAFLDA}.  For a simpler proof and fuller account of the theory of left distributive algebras, see \cite{LaM:FOGLDA}.
Using braid groups Dehornoy showed within the standard axioms of set theory that the above-mentioned orderings on the free left distributive algebra are linear
\cite{Deh:BGLDO}; Dehornoy has also contributed substantially to the literature on algebras satisfying $\Sigma$.  See, for example, \cite{Deh:BSD}.

We now move on to preliminaries regarding Borel reducibility.
Recall that a subset of a topological space is \emph{Borel}
if it lies in the least $\sigma$-algebra containing the open sets, and
that a function between two topological spaces is Borel if the inverse
image of any Borel set (or equivalently, of any open set) is Borel.
Thus, to discuss Borel reducibility between classes of countable structures,
we first define a topology on each of these classes.
We briefly sketch this definition here,
and refer the reader to Section~2.3
of Hjorth's book \cite{Hjo:COER} for further details.

We exclusively consider countable structures,
and so may assume that each structure has underlying
set $\N$. Furthermore all of the classes of structures we consider are
first-order, namely, the structures have finitely many relations and operations,
and the class is defined by formulas involving these relations and operations.
The relations and operations of a structure in one of these classes can
thus be represented by a set of tuples from $\N$.
Indeed we follow the common practice of identifying a directed graph
$(\N,E)$ (with vertex set $\N$)
with the set $\{(m,n)\st m\mathrel{E}n\}\subseteq\N^2$,
and we may identify an algebra $(\N,*)$
with the set $\{(\ell,m,n)\st \ell*m=n\}\subset\N^3$.
The space of countable structures
for a given signature with finitely many operation and relation symbols
can thus be identified with a
subset of Cantor space
via the usual identification of a power set $\Power(X)$
with the space of characteristic functions $2^X$; the set $X$ here is
a product of sets of the form $\N^k$, one for each relation and operation,
and is in particular countable.
The topology considered on these classes is the standard topology on the
Cantor space.  Note that a clopen subbase for this topology is given
by the sets defined by determining a single ``bit'' from $2^X$ ---
for example, on the space of countable algebras with underlying set $\N$,
the subbase is the
collection as $\ell, m$, and $n$ vary over $\N$ of all
sets either of the form $\{(\N,*)\st \ell*m=n\}$ or of the form
$\{(\N,*)\st \ell*m\neq n\}$.

We deviate from this conventional framework in one detail:
for expositional clarity,
the keis that we construct will have underlying set
$\N\times\{0,1\}$ rather than $\N$.  However, this discrepancy can be easily
overcome using the canonical
identification of $\N\times\{0,1\}$ with $\N$ via the map
$(n,i)\mapsto 2n+i$.

Note that the Cantor space $2^X$ with $X$ countable is a separable topological
space (that is, it has a countable dense set) and may be endowed with a complete
metric: identifying $X$ with $\N$, let $d(x,y)=2^{-n}$ where $n$ is least such
that $x(n)\neq y(n)$.
Separable, completely metrizable spaces such as $2^X$ and
$\mathbb{R}$ are known as \emph{Polish spaces}.
As outlined in the Introduction, we have the following standard definitions.
\begin{defn}
Let $X$ and $Y$ be Polish spaces, $E$ an equivalence relation on $X$, and
$F$ an equivalence relation on $Y$.  We say that $E$ is
\emph{Borel reducible to $F$}, written $E\leq_BF$, if there is a Borel
function $f$ from $X$ to $Y$ such that for all $x$ and $x'$ in $X$,
$x\mathrel{E}x'$ holds (that is, $x$ is $E$-equivalent to $x'$) if and only if
$f(x)\mathrel{F}f(x')$ holds.\\

We say that $E$ is \emph{continuously reducible} to $F$, written $E \leq_c F$, if there is a continuous function $f$ from $X$ to $Y$ such that for all $x$ and $x'$ in $X$, $x \mathrel{E} x'$ if and ony if $f(x) \mathrel{F} f(x')$.\\

If $F$ is the isomorphism relation for a first-order class of countable structures for a finite signature each with underlying set $\N$, we say $F$ is \emph{Borel complete} if every other such class has isomorphism relation Borel reducible to $F$.
\end{defn}

Continuous maps are of course Borel, and all maps we construct in the sequel will be continuous so in particular Borel.

\section{Keis are Borel Complete}

It is folklore that the class of countable irreflexive directed graphs is
Borel complete --- see Section 13.1 of Gao's book \cite{Gao:IDST} for
a proof of the stronger statement that the subclass of countable irreflexive symmetric graphs is Borel complete.
The general strategy of this section is to construct a kei from an arbitrary irreflexive directed graph, and then to show that the resulting keis are isomorphic if and only if the original graphs are isomorphic. Since the map taking each irreflexive directed graph to the corresponding kei will be Borel (indeed, continuous), this will establish that the class of countable keis is also Borel complete. To this end we shall describe
how to build what Kamada \cite{Kam:QDS} calls a \emph{dynamical quandle};
the specific dynamical quandles we construct will in fact be keis.

In all of the sequel we exclusively discuss graphs that are irreflexive and directed, but for the sake of the casual reader, we will repeat these hypotheses each time they are used.

Let $A$ be a set and $\tau$
a bijection from $A$ to itself.  
Let $\phi$ be a map from $A$ to the power set $\Power(A)$
such that for every $a\in A$,
$\phi(a)$ contains $a$, $\phi(a)$ is closed under $\tau$ and $\tau^{-1}$, and
$\phi(a)=\phi(\tau a)$.  We will refer to such maps $\phi$ as \emph{$\tau$-replete}.
Kamada observes \cite[Theorem~4]{Kam:QDS} that with the operation $*$
defined by
\[
a*b =
\begin{cases}
b&\text{if }a\in\phi(b)\\
\tau b&\text{if }a\notin\phi(b),\\
\end{cases}
\]
the structure $(A,*)$ is a quandle.
Kamada uses an equivalent definition with a function $\theta$ defined
on $\tau$-orbits
rather than our orbit-invariant function $\phi$ on elements of $A$.
Axioms (ii) and (iii) of Definition \ref{algdefns} are immediate from the
assumptions on $\phi$, and (i) follows by checking cases:
\[
a*(b*c)=(a*b)*(a*c)=
\begin{cases}
c&\text{if }a\in\phi(c)\text{ and }b\in\phi(c)\\
\tau c&\text{if }a\in\phi(c)\text{ and }b\notin\phi(c)\\
\tau c&\text{if }a\notin\phi(c)\text{ and }b\in\phi(c)\\
\tau^2c&\text{if }a\notin\phi(c)\text{ and }b\notin\phi(c).
\end{cases}
\]
Moreover, if $\tau$ is an involution, then clearly axiom (iv) also holds
and so the quandle is a kei.  Following Kamada, but using our
$\phi$ rather than Kamada's $\theta$, we call
this $(A,*)$ the \emph{quandle derived from $(A,\tau)$ relative to $\phi$}.
Kamada named the objects so constructed \emph{dynamical quandles},
in line with a
view of the pair $(A,\tau)$ as a dynamical system, and we shall call those
dynamical quandles that are keis 
\emph{dynamical keis}.

To encode an irreflexive directed graph $G=(V,E)$ into a kei $Q_G$,
we use the dynamical quandle construction with underlying set
a \emph{pair} of copies of the vertex set $V$ of $G$.
Our involution $\tau$ simply
switches between the two copies of the vertex set, and
the function $\phi$ corresponds to choosing the set of neighbours
(in one direction) for each vertex of $G$, irrespective of which
copy of $V$ the vertices lie in. 

\begin{defn}\label{keifromgraph}
Suppose $G=(V,E)$ is an irreflexive directed graph.
Let $\tau$ be the involution on $V\times\{0,1\}$ taking $(v,0)$ to $(v,1)$
and $(v,1)$ to $(v,0)$ for every $v$ in $V$.
Let $\bar\phi_G$ be the function from $V$ to $\Power(V)$ defined by
$u\in\bar\phi_G(v)$ if and only if $u\mathrel{E}v$ or $u=v$.
Let $\phi_G$ from $V\times\{0,1\}$ to $\Power(V\times\{0,1\})$ be the
function obtained from $\bar\phi_G$ by ignoring second coordinates:
$(u,i)\in\phi_G(v,j)$ if and only if $u\in \bar\phi_G(v)$, that is,
if and only if $u\mathrel{E}v$ or $u=v$.
Note that $\phi_G$ is $\tau$-replete.
The \emph{kei $Q_G$ associated to $G$} is the quandle derived from
$(V\times\{0,1\},\tau)$ relative to $\phi_G$, and we denote the operation
on $Q_G$ by $*_G$.
\end{defn}

Thus, $Q_G$ is a kei on underlying set $V\times\{0,1\}$ with
operation $*$ such that
$(u,i)*(v,j)$ equals $(v,j)$ if there is an edge from $u$ to $v$ in $G$ or if \(u=v\), and $(u,i)*(v,j)$ is $(v,1-j)$ otherwise.


We now begin toward Theorem \ref{cong}, which says that the dynamical keis $Q_G$ and $Q_{G'}$ constructed from graphs $G$ and $G'$ are isomorphic if and only if the graphs $G$ and $G'$ are isomorphic.  First we prove the existence of a particular, useful involution of the kei $Q_G$ (Lemma \ref{IW}). 

\begin{lemma}\label{IW}
For every irreflexive directed graph $G$ with underlying set $V$ and every $W\subseteq V$, the function
$I_W:Q_G\to Q_G$ 
defined by
\[
I_W(v,j)=
\begin{cases}
(v,j)&\text{if }v\in W\\
(v,1-j)&\text{ if } v \notin W
\end{cases}
\]
is an involution of $Q_G$.
\end{lemma}
\begin{proof}
By inspection $I_W$ is a bijection and moreover $(I_W)^2$ is the identity map.
To see that $I_W$ respects the quandle operation $*$ of $Q_G$, we must verify that $I_W((u,i) \ast (v,j)) = I_W(u,i) \ast I_W(v,j)$.  Note
that for each $(v,j)\in Q_G$, either both of $(u,0)$ and $(u,1)$ are in $\phi_G(v,j)$ or neither is, so \[
(u,i)*(v,j)=(I_W(u,i))*(v,j)=
\begin{cases}
(v,j)&\text{if }(u,i) \in \phi(v,j)\\
(v,1-j)&\text{ if } (u,i) \notin \phi(v,j).
\end{cases}
\]
So
\[
I_W((u,i)*(v,j))=
\begin{cases}
I_W(v,j)&\text{if }(u,i) \in \phi(v,j)\\
I_W(v,1-j)&\text{ if } (u,i) \notin \phi(v,j)
\end{cases}
\]
and 
\[
I_W((u,i))*I_W(v,j)=
\begin{cases}
I_W(v,j)&\text{if }(u,i) \in \phi(I_W(v,j)) = \phi((v,j))\\
(v,1-j) = I_W(v,1-j) &\text{ if } v \in W \text{ and } (u,i) \notin \phi((v,j))\\
(v, j) = I_W(v, 1-j) &\text{ if } v \notin W \text{ and } (u, i) \notin \phi((v,j)).
\end{cases}
\]
Thus it is established that $I_W$ is a homomorphism, indeed an involution of $Q_G$.
\end{proof}

A slicker if less direct proof of Lemma \ref{IW} is to consider the graph $G'$ on $V\,\dot{\cup}\,\{v_0\}$ (where $\dot{\cup}$ denotes disjoint union) with $G' \upharpoonright V = G$ and $v_0 \mathrel{E} v$ if and only if $v$ is in $W$ for each $v$ in $V$.  Then $Q_{G'} \upharpoonright V \times \{0,1\} = Q_G$, and $m_{v_0} \upharpoonright Q_G = I_W$.

The keis constructed in Definition \ref{keifromgraph} are in fact quite general dynamical keis.  Indeed the only extra constraint we need on dynamical keis
to get a kei $Q_G$ associated to a graph $G$ is that the involution $\tau$ has no fixed points.

\begin{defn}\label{dikei defn}
A kei $(A,*)$ is called a
\emph{\dikei}\footnote{In baking, one \emph{folds} ingredients
to achieve complete mixing with minimal disruption.}
if there is an involution $\tau$ of $A$ with no
fixed points and a $\tau$-replete function $\phi$ 
such that $(A,*)$ is the
quandle derived from $(A,\tau)$ relative to $\phi$.
\end{defn}

By definition the kei $Q_G$ associated to any graph $G$ is a folded kei.
As alluded to above, we also have a converse to this.

\begin{prop} \label{Qg is graph kei}
Every \dikei is isomorphic to a kei of the form $Q_G$ for some
irreflexive directed graph $G$.
\end{prop}
\begin{proof}
Let $(A,*)$ be a \dikei, and in particular
suppose $(A,*)$ is the quandle derived from $(A,\tau)$ relative to $\phi$
for $\tau$ an involution of $A$ without fixed points and $\phi$ a $\tau$-replete function from $A$ to $\Power(A)$.
Choose a subset $V$ of $A$ such that
for each pair $\{a,\tau a\}$ of elements of $A$,
exactly one of $a$ and $\tau a$ is in $V$,
and express $A$ as the disjoint union $A=V\cup\{\tau v\st v\in V\}$.
For each $v$ in $V$, let $\bar{\phi}(v)$ denote the set $\phi(v)\cap V$;
since $(A,*)$ is the quandle derived from $(A,\tau)$ relative to $\phi$
we have that $\bar{\phi}(v)$ is the set of $u$ in $V$ such that $u*v=v$
(this $\bar\phi$ will be $\bar{\phi}_G$ as in Definition~\ref{keifromgraph}
for the graph $G$ we now construct).
Take the directed graph $G$ on vertex set $V$ with edge relation defined by
$u\mathrel{E}v$ if and only if $u\in\bar{\phi}(v)$ holds.
Then it is straightforward to check that the map from $Q_G$ to $A$ taking
$(v,0)$ to $v$ and $(v,1)$ to $\tau v$ is an isomorphism of keis.
\end{proof}

We will now state the main result.

\begin{thm}\label{cong}
For irreflexive directed graphs $G$ and $G'$ and the associated keis $Q_G$ and $Q_G'$, $G \cong G'$ if and only if $Q_G \cong Q_{G'}$
\end{thm}

\begin{proof}
One direction is a fairly straightforward observation:
\begin{remark*}\label{isographgivesisokei}
Isomorphic irreflexive directed graphs have isomorphic associated keis.
\end{remark*}
\begin{proof}[of Remark]
Recall that a graph isomorphism is a bijection between vertices that preserves both the edge relation and the failure of the edge relation.  Given graphs $G = (V, E)$ and $G' = (V', E')$ with an isomorphism $h:G \rightarrow G'$ between them, $u\mathrel{E}v$ in $G$ if and only if $h(u)\mathrel{E'}h(v)$ in $G'$, so $u$ is in $\bar{\phi}_G(v)$ if and only if $h(u)$ is in $\bar{\phi}_{G'}(h(v))$. Therefore by construction of the quandles $Q_G$ and $Q_G'$, $h$ induces an isomorphism $h_Q$ from $Q_G$ to $Q_G'$ taking $(u,i)$ to $(h(u),i)$.  Indeed for vertices $u$ and $v$ in $G$, we have that $(u,i) \in \phi_G(v,j)$ holds if and only if $(h(u),i) \in \phi_{G'}(h(v),j)$ holds. 
The verification that $x \ast_G y = z$ if and only if $h_Q(x) \ast_{G'} h_Q(y) = h_Q(z)$ follows immediately.
\end{proof}

For the converse, we will show that any two isomorphic keis of the form $Q_G$ and $Q_{G'}$ admit an isomorphism induced by an isomorphism of the underlying graphs $G$ and $G'$. Not all kei isomorphisms between $Q_G$ and $Q_{G'}$ arise from graph isomorphisms; indeed, Lemma~\ref{IW} gives continuum many others. Also, if the graph $K$ is the complete irreflexive directed graph on $V$, then $Q_K$ is the trivial kei on $V\times\{0,1\}$, with
$(u,i)*(v,j)=(v,j)$ for all $(u,i)$ and $(v,j)$.
Of course there are many automorphisms of the trivial kei that are not of the form given by Lemma~\ref{IW} or induced by a graph isomorphism:
any permutation of the underlying set $V\times\{0,1\}$ is an automorphism of this kei.
We will see in the Claim that follows that any kei isomorphism $\rho$
between folded keis splits into two parts, one
of the type described by Lemma \ref{IW}
and one given by an automorphisms of a trivial kei.
Each of these can be converted into a partial isomorphism of the desired form,
and the pieces recombined to yield the graph isomorphism required for the
Theorem.

To aid with intuition, for any graph $G=(V,E)$ with associated kei
$Q_G=(V\times\{0,1\},*_G)$, we refer to $V\times\{0\}\subset Q_G$ as the
\emph{bottom} of $Q_G$ and $V\times\{1\}\subset Q_G$ as the \emph{top} of
$Q_G$.  Also for any $v$ in $V$ we refer to each of $(v,0)$ and $(v,1)$ as
the \emph{twin} of the other.

\begin{claim*}\label{main}
Suppose $G=(V_G,E_G)$ and $G'=(V_{G'},E_{G'})$ are irreflexive directed graphs such that there is a kei isomorphism $\rho$ with $\rho\from Q_G\to Q_{G'}$.
Then there is bijection $f$ from $V_G \to V_{G'}$
such that, viewed as a map from $G$ to $G'$, $f$ is a graph isomorphism.
\end{claim*}

\begin{proof}[Proof of Claim]
For any graph $H=(V,E)$, we split
the underlying set $V$ into two components, which we call the
``fixed points'' and the ``moving points'' based on their behaviour in the
quandle $Q_H$.
The purely graph-theoretic definitions of the fixed points and moving points is simpler, so we give them first: the fixed points are those which are complete for inward edges, and the
moving points are those that are not.  That is,
\[
F_H=\{v\in V\st \forall u\in V (u\mathrel{E}v)\}.
\]
From the quandle point of view,
the fixed points may equivalently be defined as
those $v$ for which left multiplication by any element of $Q_H$
does not swap $(v,0)$ with $(v,1)$, that is,
\[
F_H=\{v\in V \st \forall (u,i)\in Q_H [(u,i)*_H(v,0)=(v,0)]\}.
\]
The moving points are then those not in $F_H$, that is,
$M_H=V\smallsetminus F_H$.

We shall define the function $f:V_G\to V_{G'}$ piecewise, giving
separately the restrictions of $f$ to the fixed points
$F_G$ and the moving points $M_G$.
In fact, these restrictions will themselves be bijections from
$F_G$ to $F_{G'}$ and from $M_{G}$ to $M_{G'}$, as is clearly necessary for
$f$ to be a graph isomorphism.

We are given an isomorphism $\rho:Q_G\to Q_{G'}$.
Let us denote by $\rho_V(v,i)$ and $\rho_I(v,i)$ respectively
the first and second components of $\rho(v,i)$: that is,
$\rho(v,i)=(\rho_V(v,i),\rho_I(v,i))$.

First we define $f$ on the moving points.
If $v$ is in $M_G$, then there is some $(u,i)$ in $Q_G$ that moves $(v,0)$.
That is, the vaule of $(u,i)*_G(v,0)$ is not $(v,0)$, and hence by the definition of $*_G$ it must be that $(u,i)*(v,0)$ is $(v,1)$, and furthermore that
$(u,i)*(v,1)$ is $(v,0)$.
Applying the kei isomorphism $\rho$ we have that $\rho(u,i)*\rho(v,0)=\rho(v,1)$ holds, and by injectivity $\rho(v,1)\neq\rho(v,0)$.
By the definition of $*_{G'}$, the first components of $\rho(v,0)$ and $\rho(v,1)$ must be equal.  We take $f(v)$ to be this value:
$f(v)=\rho_V(v,0)=\rho_V(v,1)$.

Clearly $f\restr M_G$ so defined is injective since $\rho$ is a bijection.
Moreover $f\restr M_G$ surjects onto $M_{G'}$.
Indeed, for $w$ in $M_{G'}$ and $(t,i)$ in
$Q_{G'}$ such that $(t,i)*_{G'}(w,0)\neq (w,0)$, we have
$\rho^{-1}(t,i)*_G\rho^{-1}(w,0)\neq \rho^{-1}(w,0)$, and so
the first component of $\rho^{-1}(w,0)$ lies in $M_G$ and has image
$w$ under $f$.

To complete the definition of $f$
it remains to give the value of  $f(v)$ for those $v$ in $F_G$.
Let $v_0$ be an element of
$F_G$.  Unlike for elements of $M_G$, it need not
be the case that $\rho_V(v_0,0)$ is the same as $\rho_V(v_0,1)$.
However, since $\rho$ is surjective, we may find
$v_{1}$ in $F_G$ and
$i_{v_1}$ in $\{0,1\}$ such that $\rho_V(v_{1},i_{v_1})=\rho_V(v_0,1)$ and $\rho_I(v_1,i_{v_1})=1-\rho_I(v_0,1)$: that is,
if $\rho(v_0,1)$ is on the bottom of the kei then $(v_1,i_{v_1})$ is chosen such that $\rho(v_1,i_{v_1})$ is its twin on the top, and conversely
if $\rho(v_0,1)$ is on the top of the kei then $(v_1,i_{v_1})$ is chosen
such that $\rho(v_1,i_{v_1})$ is its twin on the bottom.
Likewise we may find
$v_{-1}$ in $F_G$ and
$i_{v_{-1}}$ in $\{0,1\}$ such that $\rho_V(v_{-1},1-i_{v_{-1}})=\rho_V(v_0,0)$
and $\rho_I(v_0,0)=1-\rho_I(v_{-1},1-i_{v_-1})$.
We may inductively extend our definitions, 
obtaining for all $k$ in $\Z$ a vertex $v_k$ in $V_G$ and $i_{v_k}$ in $\{0,1\}$ (with $i_{v_0}=0$) such
that $\rho_V(v_k,1-i_{v_k})=\rho_V(v_{k+1},i_{v_{k+1}})$.
Note that if there is some $k$ such that $v_k=v_0$, then
$i_{v_k}$ defined in this way will be equal to $i_{v_0}$, so our notation
$i_{v_j}$ gives a well-defined function from vertices $v_j$ in $F_G$ to
members of $\{0,1\}$.  Indeed, (construing for now $i_{v_j}$ as a function
of $j$ rather than $v_j$) consider the first repetition in the sequence
$(v_0,i_{v_0}), (v_0, 1-i_{v_0}), (v_1,i_{v_1}),\ldots$.
Clearly if $(v_k,i_{v_k})$ is distinct from all of its predecessors in the
sequence, then so too is $(v_k,1-i_{v_k})$.
Thus, the first repetition in the sequence must be of the form
$(v_k,i_{v_k})$.  If $(v_k,i_{v_k})=(v_j,1-i_{v_j})$ for some $j<k$,
then of course $\rho(v_k,i_{v_k})=\rho(v_j,1-i_{v_j})$, so swapping betweeen the top and bottom of the kei, we have from the inductive construction that
$\rho(v_{k-1},1-i_{v_{k-1}})=\rho(v_{j+1},i_{v_{j+1}})$.
But then by the minimality of $k$ as giving a repetition, we must have
$j=k-1$, so $(v_k,i_{v_k})=(v_{k-1},1-i_{v_{k-1}})$, violating the fact from
the construction that $\rho(v_k,i_{v_k})\neq\rho(v_{k-1},1-i_{v_{k-1}})$.

The set $\{v_j\st j\in\Z\}$ may be finite or infinite,
but the corrresponding subset $\{\rho_V(v_j,i_{v_j})\st j\in\Z\}$ has the
same cardinality: $(v_j,i_{v_j})=(v_k,i_{v_k})$ if and only if $\rho(v_j,i_{v_j})=\rho(v_k,i_{v_k})$.
Note also that for each $k$,
the left multiplication maps $m_{\rho(v_k,1-i_{v_k})}$ and
$m_{\rho(v_{k+1},i_{v_{k+1}})}$ on
$Q_{G'}$ are the same
since $\rho(v_k,1-i_{v_k})$ and $\rho(v_{k+1},i_{v_{k+1}})$ have the same
first component.  Therefore
$m_{(v_k,1-i_{v_k})}$ and $m_{(v_{k+1},i_{v_{k+1}})}$ are the same on $Q_G$.
It follows that $v_k$ and $v_{k+1}$ have outward edges to the same
other vertices in $G$, as well as to each other, and by induction the same is
true of all members of the set $\{v_k\st k\in\Z\}$;
likewise, all members of the set $\{\rho_V(v_k,i_{v_k})\}$
have edges to one another and to the same other vertices.

The set $F_G$ may be partitioned into such ``cycles'' of vertices
$\{v_k\st k\in\Z\}$ by choosing a
starting vertex $v_0$ in each cycle.  With such choices made,
we in particular have an assignment of $i_{v}$ in $\{0,1\}$ to each $v$ in $F_G$,
and may define $f\restr F_G$ by $f(v)=\rho_V(v,i_v)$.
Clearly with this definition $f\restr F_G$ is a bijection
from $F_G$ to its image.  Moreover its image is all of $F_{G'}$:
if $(t,i)*_{G'}(w,0)= (w,0)$ for all $(t,i)$ in $Q_{G'}$, then
$\rho^{-1}(t,i)*_G\rho^{-1}(w,0)=\rho^{-1}(w,0)$ for all $(t,i)$ in $Q_{G'}$,
that is, $(u,j)*_G\rho^{-1}(w,0)=\rho^{-1}(w,0)$ for all $(u,j)$ in $Q_G$.

We have thus constructed a bijection $f:V_G\to V_{G'}$, and it remains to
show that $f$ is in fact a graph isomorphism from $G$ to $G'$.
So let $u$ and $v$ be vertices of $G$.
If $v$ is in $F_G$, then $f(v)$ is in $F_{G'}$, so both $u\mathrel{E_G}v$ and
$f(u)\mathrel{E_{G'}}f(v)$ hold.
Suppose $v$ is in $M_G$.
If $u$ is in $F_G$ we have $i_u$ in $\{0,1\}$ as defined above, and otherwise
take $i_u=0$.  Then
\[
(u,i_u)*_G(v,0)=\begin{cases}
(v,0)&\text{if }u\mathrel{E_G}v\text{ or } u=v\\
(v,1)&\text{otherwise,}
\end{cases}
\]
so
\[
\rho(u,i_u)*_G\rho(v,0)=\begin{cases}
\rho(v,0)&\text{if }u\mathrel{E_G}v\text{ or } u=v\\
\rho(v,1)&\text{otherwise.}
\end{cases}
\]
Since the first component of $\rho(u,i_u)$ is $f(u)$ and the first component
of $\rho(v,0)$ is $f(v)$, we have that
$f(u)\mathrel{E_G'}f(v)$ if and only if $u\mathrel{E_G}v$,
completing the proof that $f$ is a graph isomorphism from $G$ to $G'$.
\end{proof}

With the Claim we have shown that, whilst not every isomorphism of keis
$Q_G$ and $Q_{G'}$ need arise from a graph isomorphism, such an isomorphism
can be used to define a graph isomorphism of $G$ and $G'$, which
by the Remark gives rise to a (potentially different) isomorphism of
$Q_G$ and $Q_G'$.
This completes the proof of Theorem~\ref{cong}.
\end{proof}

\begin{thm}
The classes of keis, quandles, racks, left distributive algebras, and
algebras satisfying $\Sigma$ are each Borel complete.
\end{thm}
\begin{proof}
Implicit in the statement that these classes of structures are Borel complete
is that we are considering the classes of countable such structures
with underlying set $\N$,
with each class topologized as described in Section~\ref{prelims}.

The map $G\mapsto Q_G$ from the class of graphs to the class of keis
is not only Borel but in fact continuous.
Recall from Section \ref{prelims} that the subbasic open sets in the space of graphs are of the form either $\{G\st m\mathrel{E}n\}$ or $\{G\st m\not\mathrel{E} n\}$.  Similarly, for quandles with underlying set $\N$, the subbasic
open sets are of the form $\{(\N, \ast)\st u \ast v = w\}$ or $\{(\N, \ast)\st u \ast v \neq w\}$.  Then by the construction of our dynamical keis,
it is clear that the inverse image of any open set is open (as we defined $\ast$ in terms of the edge relation of $E$).
Hence the map taking $G$ to $Q_G$ is continuous and so certainly Borel,
and therefore because the class of graphs is Borel complete,
we have shown that the
keis, and hence quandles, and hence racks, and hence left distributive algebras are Borel complete.

Because the language of $\Sigma$ is different from that of left distributive algebras, a different argument is needed to show that the class
of algebras satisfying $\Sigma$ is Borel complete.
For this we utilize the result of Mekler~\cite{Mek:SNG}
that the class of groups is Borel complete (see \cite[\S 2.3]{FS:BRCCS}
for a sketch of the argument). As discussed after Definition~\ref{LDmonoiddefn}, every group endowed with its conjugation operation and its
group operation satisfies $\Sigma$.  The inclusion map
$(G,\circ)\mapsto(G,\circ,*)$ where $\circ$ denotes the group operation and
$*$ denotes conjugation is easily seen
to be continuous and so is certainly Borel.
Of course, since the group operation
is one of the two operations in the language of $\Sigma$,
and the other is conjugation which is
determined by the group operation, two groups are isomorphic if and only if their corresponding structures satisfying $\Sigma$ are isomorphic.
We thus have that group isomorphism Borel reduces to isomorphism as
algebras satisfying $\Sigma$, and therefore that the latter is Borel complete.
\end{proof}

\section{Concluding remarks}

We have shown that in the Borel reducibility sense,
the class of left distributive algebras is as complex as possible.
Another formalization of the question of complexity is in a
category-theoretic setting.
Just as the class of graphs is maximal in the Borel completeness sense
(and indeed our proof made use of this fact),
the \emph{category} of graphs is universal
in the sense that every algebraic category fully embeds into it
\cite[Theorem~5.3]{PuT:CATRGSC}.
There are many such universality results for other categories
--- see, for example, \cite{PuT:CATRGSC} ---
raising the following natural question.

\begin{question}
Does the category of graphs fully embed into the category of left distributive
algebras?
\end{question}
Of course the same question may also be asked of the category of racks, the
category of quandles, and the category of keis.  We note that the construction
of $Q_G$ from $G$ in Theorem~\ref{cong} is not even functorial, since a
graph homomorphism need not preserve non-edges.
Potentially an even more problematic obstacle, however, is the fullness
requirement --- we have seen that dynamical
keis admit many more homomorphisms than simply those arising from graph
homomorphisms, at least in our construction.
On the other hand, even if it turns out that
the category of graphs cannot be fully embedded into the category of keis
because keis always admit many homomorphisms,
there may be interesting minimal-non-fullness, maximal-complexity
results to be obtained in this direction.
As an analogy, there can be no full embedding of the category of graphs
into the category of abelian groups,
as any two abelian groups $A$ and $B$
admit at least one homomorphism between them
(the 0 map) and the set of homomorphisms between them $\Hom(A, B)$ naturally
forms an abelian group.
Nevertheless
Prze\'zdziecki \cite{Prz:GrAb} has shown that there is an embedding
$\calA$ from
the category of graphs to the category of abelian groups such that
$\Hom(\calA G,\calA G')$ is the free abelian group generated by $\Hom(G,G')$
--- the best possible result given these constraints.

As mentioned in the the introduction, the implication operation in a Boolean Algebra is left distributive.  Borel completeness of the isomorphism relation on Boolean algebras was proven by Camerlo and Gao in \cite{CGao:CIRCBA} and does not follow from what we have proven.  Their work shows that a classification of countable Boolean algebras due to Ketonen uses objects for the complete invariants that ``cannot be improved in an essential way'' \cite{CGao:CIRCBA}.

In contrast, our main result is that the class of quandles is Borel complete while tame knots are trivial in terms of Borel reducibility.  Whilst the subclass of finitely presented quandles contains the quandles associated with all tame knots and is itself trivial in this context, it is not clear that this finite presentability constraint can be used in practice to simplify the quandle isomorphism problem. Thus, our result suggests that there may well exist a more practical complete invariant for tame knots, with an isomorphism problem that is not as difficult as that for quandles.

\bibliography{ABT}

%
%
%
%
%
%

\affiliationone{
   Andrew D. Brooke-Taylor\\
   University of Bristol\\
   School of Mathematics\\
   Howard House\\
   Queen's Avenue\\
   Bristol  BS8 1SN\\
   United Kingdom
   \email{a.brooke-taylor@bris.ac.uk}}
\affiliationtwo{
   Sheila K. Miller\\
   City University of New York\\
   New York City College of Technology\\
   Department of Mathematics\\
   300 Jay Street\\
   Brooklyn, NY 11201\\
   USA
   \email{smiller@citytech.cuny.edu}}
%
\end{document}